\documentclass[12pt,twoside]{amsart}
\usepackage{amsmath, amsthm, amscd, amsfonts, amssymb, graphicx}
\usepackage[bookmarksnumbered, plainpages]{hyperref}

\usepackage{mathrsfs}

\newcommand{\tr}{\text{tr}}

\newtheorem{thm}{Theorem}[section]
\newtheorem{cor}[thm]{Corollary}
\newtheorem{lem}[thm]{Lemma}
\newtheorem{prop}[thm]{Proposition}
\newtheorem{defn}[thm]{Definition}
\newtheorem{rem}[thm]{\bf{Remark}}

\numberwithin{equation}{section}

\def\tr{\mbox{tr}}

\begin{document}

\title{\bf On the existence of harmonic metrics on non-Hermitian Yang-Mills bundles}
\author{Changpeng Pan, Zhenghan Shen, Xi Zhang}

\address{Changpeng Pan\\School of Mathematics and Statistics\\
Nanjing University of Science and Technology\\
Nanjing, 210094,P.R. China\\ }
\email{mathpcp@njust.edu.cn}

\address{Zhenghan Shen\\School of Mathematics and Statistics\\
Nanjing University of Science and Technology\\
Nanjing, 210094,P.R. China\\}\email{mathszh@njust.edu.cn}

\address{Xi Zhang\\School of Mathematics and Statistics\\
Nanjing University of Science and Technology\\
Nanjing, 210094,P.R. China\\ } \email{mathzx@ustc.edu.cn}

\subjclass[2020]{53C07, 57N16}
\keywords{non-Hermitian Yang-Mills connection, harmonic metric, heat flow}

\maketitle

\begin{abstract}
In this paper, we study the non-Hermitian Yang-Mills (NHYM for short) bundles over compact K\"ahler manifolds. We show that the existence of harmonic metrics is equivalent to the semisimplicity of NHYM bundles, which confirms the Conjecture 8.7 in \cite{KV}.

\end{abstract}

\vskip 0.2 true cm


\pagestyle{myheadings}
\markboth{\rightline {\scriptsize C. Pan et al.}}
         {\leftline{\scriptsize On the existence of harmonic metrics on NHYM bundles}}

\bigskip
\bigskip


\section{ Introduction}
Let $(X,\omega)$ be a K\"ahler manifold of dimension $n$, $(E,D)$ be a complex flat vector bundle over $X$. Given a Hermitian metric $H$ on $E$, there is a unique decomposition
\begin{equation}
D=D_{H}+\psi_{H},
\end{equation}
where $D_{H}$ is an $H$-unitary connection and $\psi_{H}\in \Omega^{1}(\mbox{End}(E))$ is self-adjoint with respect to $H$.
By the Riemann-Hilbert correspondence, we know that there is a one-to-one correspondence between the flat bundle $(E,D)$ and the fundamental group representation $\rho_{D}:\pi_{1}(X)\rightarrow \mbox{Gl}(r,\mathbb{C})$. Then any Hermitian metric $H$ on $E$ induces a $\rho_{D}$-equivariant map
\begin{equation}
    f_{H}:\tilde{X}\rightarrow \mbox{Gl}(r,\mathbb{C})/\mbox{U}(r),
\end{equation}
where $\tilde{X}$ is the universal covering space of $X$. A Hermitian metric $H$ is called harmonic if $f_{H}$ is a  harmonic map.  On the other hand, we know that $\psi_{H}=-\frac{1}{2}f_{H}^{-1}df_{H}$ (\cite{Gui}). So $H$ is a harmonic metric if and only if it is a critical point of the following energy functional
\begin{equation}\label{eq:f}
    E(D,H)=\frac{1}{2}\int_{X}|\psi_{H}|^{2}_{H,\omega}\frac{\omega^{n}}{n!}.
\end{equation}
That is, it satisfies the Euler-Lagrange equation
\begin{equation}
	D_{H}^{*}\psi_{H}=0.
\end{equation}

The notion of harmonic maps was introduced by Eells and Sampson (\cite{ES}) in the 1960s and generalized to harmonic metric on flat bundles by Corlette (\cite{Cor}). Donaldson (\cite{Don3}) and Corlette (\cite{Cor}) proved that $(E,D)$ admits a harmonic metric if and only if it is semisimple. This result has also been extended to some noncompact or non-K\"ahler manifolds by Jost-Zuo (\cite{JZ1,JZ2}), Simpson (\cite{S3}), Mochizuki (\cite{M2}), Collins-Jacob-Yau (\cite{CJY}), Pan-Zhang-Zhang (\cite{PZZ}) and Wu-Zhang (\cite{WZ}).

On the other hand, we have the Hitchin-Kobayashi correspondence which states that the stability of holomorphic  vector bundles is equivalent to the existence of irreducible Hermitian-Einstein metrics. It was proved by Narasimhan-Seshadri (\cite{NarSe}) for Riemann surfaces, Donaldson (\cite{Don2,Don1}) for algebraic manifolds and Uhlenback-Yau (\cite{UhYau}) for higher dimension K\"ahler manifolds. Hitchin (\cite{Hit}) and Simpson (\cite{S1}) proved the Hitchin-Kobayashi correspondence for Higgs bundles. Based on the results of Corlette (\cite{Cor}), Donaldson (\cite{Don3}), Hitchin (\cite{Hit}) and Simpson (\cite{S1,S3}), one can obtain the famous non-abelian Hodge correspondence between the moduli space for semisimple flat bundles and the moduli space for poly-stable Higgs bundles with vanishing Chern numbers. There are many important and interesting generalizations of the non-abelian Hodge correspondence, such as: Jost and Zuo (\cite{JZ2}), Biquard and Boalch (\cite{BB}), Mochizuki (\cite{M1,M2}) for the quasiprojective varieties case; Bradlow, Garc\'ia-Prada and Mundet i Riera (\cite{BGM}), Garc\'ia-Prada and Mundet i Riera (\cite{GM1}), Garc\'ia-Prada, Gothen and I.Mundet i Riera (\cite{GGM}) for the the principal $G$-Higgs bundles case; Biswas and Kasuya (\cite{BK1,BK2}) for Sasakian manifolds case.

\medskip
A connection $D$ in $E$ is called non-Hermitian Yang-Mills (NHYM for short) if its curvature $F_D$
satisfies
\begin{equation}\label{DFNH}
	F_{D}=D^{2}\in \Omega^{1,1}_{X}(\mbox{End}(E))\ \ \ \text{and} \ \ \sqrt{-1}\Lambda_{\omega}F_{D}=\lambda Id_{E},
\end{equation}
where $\lambda\in\mathbb{R}$ is a constant. $(E,D)$ is called a  NHYM bundle if $D$ is a NHYM connection. This concept was first introduced by Kaledin and Verbitsky in \cite{KV}. Their original motivation is to extend the Corlette-Donaldson-Hitchin-Simpson's correspondence to vector bundles with arbitrary Chern class. In their paper, they studied the moduli space of NHYM connections $\mathscr{M}$, and they proved that $\mathscr{M}$ is naturally complex analytic with dimension twice that of the subspace of Hermitian Yang-Mills connections $\mathscr{M}_{0}$. They also constructed a complex symplectic $2$-form near $\mathscr{M}_{0}$, and they conjectured that this form gives rise to a hyper-K\"ahler structure on the whole of $\mathscr{M}$.  At the end of their paper, the authors listed some open problems and conjectures which generalize the known facts about flat bundles to non-Hermitian Yang-Mills bundles.

In analogy with the flat bundle case, we define the harmonic metrics on NHYM bundles as follows.
\begin{defn}
	Let $(E,D)$ be a NHYM bundle, a Hermitian metric $H$ is called harmonic if it is a critical point of the energy functional (\ref{eq:f}), i.e., $D_{H}^{*}\psi_{H}=0$.
\end{defn}

The main work of this paper is to prove the following theorem which gives the answer of the Conjecture 8.7 in \cite{KV}.
\begin{thm}\label{thm:m}
	Let $(X,\omega)$ be a compact K\"ahler manifold, $(E,D)$ be a NHYM bundle over $X$. Then $(E,D)$ admits a harmonic metric if and only if it is semisimple.
\end{thm}
In this paper, we use the heat flow approach to prove the existence of harmonic metrics. We study the following equation on the NHYM bundle $(E,D)$,
\begin{equation}
	\left\{\begin{split}
	&H^{-1}(t)\frac{\partial H(t)}{\partial t}=2D_{H(t)}^{*}\psi_{H(t)},\\
	&H(0)=H_{0}.
	\end{split}\right.
\end{equation}
The NHYM conditions ensures that the heat flow is a parabolic flow and therefore a short-time solution exists. The long-time existence of the flow can also be proved by a classical method. We also introduce the concept of Donaldson's functional $\mathcal{M}(K,H)$ on NHYM bundles, and prove that it is decreasing along the flow. These concepts are widely used in the proofs of the Hitchin-Kobayashi correspondence. It is important to note that the NHYM condition is critical in all relevant calculations.

This paper is organized as follows. In Section \ref{sec:Pre}, we introduce some basic concepts and show that the existence of harmonic metrics implies the semisimplicity. In Section \ref{sec:DF}, we introduce the Donaldson type functional on NHYM bundles. In Section \ref{sec:FL}, we introduce a heat flow and prove the long time existence of such flow. In Section \ref{sec:FP}, we complete the remainder proof of Theorem \ref{thm:m}.

\section{Preliminary}\label{sec:Pre}
Let $(X,\omega)$ be a K\"ahler manifold of dimension $n$, $E$ be a complex vector bundle of rank $r$ over $X$. If $D$ is a NHYM connection on $E$, then we call $(E,D)$ a NHYM bundle. By the definition, the $(0,1)$-part of $D$ determines a holomorphic structure in $E$. Let $\mathcal{E}$ be the sheaf of holomorphic sections. 
In analogy with the flat bundle case, we can define the simplicity of $(E,D)$.

\begin{defn}
A subbundle $S\subset E$ is called $D$-subbundle of $(E,D)$ if it is a $D$-invariant subbundle of $E$.
\end{defn}

\begin{defn}
 We say the NHYM bundle $(E,D)$ is simple if there is no proper $D$-subbundle. We say $(E,D)$ is semisimple if it is the direct sum of several simple NHYM bundles.
\end{defn}

\begin{lem}\label{l:0}
	Let $\mathcal{S}$ be a coherent subsheaf of $\mathcal{E}$, $S$ be the holomorphic bundle on $X\setminus \Sigma$ corresponding to $\mathcal{S}$, where $\Sigma$ is the singularity set of $\mathcal{S}$. If $S$ is a $D$-subbundle of $E|_{X\setminus\Sigma}$, then there exists a $D$-subbundle $\bar{S}$ of $E$ on $X$ such that $\bar{S}|_{X\setminus \Sigma}=S$.	
\end{lem}
\begin{proof}
	Let $H$ be a Hermitian metric on $E$, $\pi$ be the projection map of $E|_{X\setminus\Sigma}$ to $S$ corresponding to $H$. Then
	\begin{equation}
		\pi^{2}=\pi^{*H}=\pi, \ \ (Id-\pi)D\pi=0 \ \ on\ X\setminus\Sigma.
	\end{equation}
Moreover, we have $\pi\in L_{1}^{2}(X,\mbox{End}(E))\cap C^{\infty}(X\setminus \Sigma,\mbox{End}(E))$, and  	\begin{equation}
	\pi^{2}=\pi^{*H}=\pi, \ \ (Id-\pi)D\pi=0\ \ a.e.\ on\ X,
\end{equation}
where $D\pi\in L^{2}(X,\mbox{End}(E))$ is the weak derivative of $\pi$. Applying the operator $D$ to the equality $\pi^{2}=\pi$ and combining with $(Id-\pi)D\pi=0$, we obtain
\begin{equation}
	D\pi\pi=0\ \ a.e.\ on\ X.
\end{equation}
Let $D=D_{H}+\psi_{H}$, where $D_{H}$ is a unitary connection and $\psi_{H}\in \Omega^{1}(\mbox{End}(E))$ is self-adjoint with respect to $H$. Set $\bar{D}=D_{H}-\psi_{H}$, then
\begin{equation}
	\bar{D}\pi(Id-\pi)=0\ \ a.e.\ on\ X.
\end{equation}
Therefore, we have
\begin{equation}\label{eq}
	\begin{split}
	D\pi=&\bar{D}\pi\pi+2[\psi_{H},\pi]\\
	=&D\pi\pi-2[\psi_{H},\pi]\pi+[\psi_{H},\pi]\\
	=&2\pi\psi_{H}(Id-\pi).
	\end{split}\ \ \ \ a.e.\ on\ X.
\end{equation}
Since $|\pi|_{H}^{2}=\mbox{rank}(S)$  almost everywhere on $X$, $\pi\in L^{p}(X,\mbox{End}(E))$ for $0<p\leq +\infty$. Then $D\pi\in L^{p}(X,\mbox{End}(E))$ for $0<p\leq +\infty$ by (\ref{eq}). Applying Sobolev's embedding theorem and the equation (\ref{eq}), we obtain that $\pi\in C^{\infty}(X,\mbox{End}(E))$. Let $\bar{S}=\pi(E)$, then the proof is complete.
\end{proof}
\begin{rem}
The above Lemma means that a $D$-invariant subsheaf can always be extended to a $D$-invariant subbundle. So our definition of $D$-simplicity of NHYM bundles coincides with Kaledin-Verbitsky's $D$-stablity in \cite{KV}.
\end{rem}

\begin{prop}\label{p:un}
	Let $\mbox{End}(E,D)$ be the endomorphism space of $(E,D)$. Assume $(E,D)$ is simple, then $\mbox{End}(E,D)\simeq \mathbb{C}$.
\end{prop}
\begin{proof}
	Let $f\in \mbox{End}(E,D)$. If $f\neq0$, then $f$ is isomorphism. Otherwise $\mbox{Ker}(f)$ or $\mbox{Im}(f)$ determines a proper $D$-subbundle of $(E,D)$ by Lemma \ref{l:0}. Since $f$ is holomorphic with repect to $D^{0,1}$, so the eigenvalues of $f$ are constants. Let $\lambda$ be a eigenvalue of $f$, then $\tilde{f}=f-\lambda Id_E$ is not a isomorphism. So $f=\lambda Id_E$.
\end{proof}

Let $(X,\omega)$ be a K\"ahler manifold of dimension $n$, $E$ be a complex vector bundle over $X$ with rank $r$. Given any connection $D$ and Hermitian metric $H$ on $E$, there is a unique decomposition
\begin{equation}
  D=D_H+\psi_H,
\end{equation}
where $D_H$ is a unitary connection preserving $H$ and $\psi_H \in \Omega^{1}(\mbox{End}(E))$ is self-adjoint with respect to $H$. If $D$ is a NHYM connection,  then
\begin{equation}
  \lambda Id_E=\sqrt{-1}\Lambda_{\omega}F_D=\sqrt{-1}\Lambda_{\omega}(D_H^2+\psi_H\wedge \psi_H+D_H\circ \psi_H+\psi_H\circ D_H).
\end{equation}
Considering the anti-self-adjoint and self-adjoint parts of the above identity, we have
\begin{equation}\label{SAP}
  \sqrt{-1}\Lambda_{\omega}D_H\psi_H=0,
\end{equation}
and
\begin{equation}\label{ASAP}
  \sqrt{-1}\Lambda_{\omega}(D_H^2+\psi_H\wedge \psi_H)=\lambda Id_E.
\end{equation}
Furthermore, if we consider the decomposition of $D_H$ (resp. $\psi_H$) into the $(1,0)$-part $\partial_H$ (resp. $\psi_H^{1,0}$) and $(0,1)$-part $\bar{\partial}_H$ (resp. $\psi_H^{0,1}$), we have
\begin{equation}
    D_{H}=\partial_{H}+\bar{\partial}_{H},\ \ \psi_{H}=\psi_{H}^{1,0}+\psi_{H}^{0,1}.
\end{equation}
Let
\begin{equation}
    D^{''}_{H}=\bar{\partial}_{H}+\psi_{H}^{1,0}, \ \ D_{H}^{'}=\partial_{H}+\psi_{H}^{0,1},
\end{equation}
then the pseudo-curvature $G_H$ with respect to the connection $D$ and  the Hermitian metric $H$ is defined as follows:
\begin{equation}
  G_H=(D_H^{''})^2=(\bar{\partial}_{H}+\psi_{H}^{1,0})^2.
\end{equation}
By the K\"ahler identities, we obtain
\begin{equation}
D^{*}_{H}\psi_{H}=2\sqrt{-1}\Lambda_{\omega}G_{H}.
\end{equation}
Therefore $H$ is a harmonic metric if and only if it satisfies the equation $\sqrt{-1}\Lambda_{\omega}G_{H}=0$.

Furthermore, if we consider the decompositions of the formulas (\ref{SAP}) and (\ref{ASAP}) into $(1,0)$-part and (0,1)-part, we obtain the following equations
\begin{equation}\label{eq:111}
\left\{\begin{split}
&\partial_{H}^{2}+\psi_{H}^{1,0}\wedge\psi_{H}^{1,0}=\bar{\partial}_{H}^{2}+\psi_{H}^{0,1}\wedge\psi_{H}^{0,1}=0,\\
&\partial_{H}\psi_{H}^{1,0}=\bar{\partial}_{H}\psi_{H}^{0,1}=0,\\
&\sqrt{-1}\Lambda_{\omega}(\partial_{H}\psi_{H}^{0,1}+\bar{\partial}_{H}\psi_{H}^{1,0})=0,\\
&\sqrt{-1}\Lambda_{\omega}([\partial_{H},\bar{\partial}_{H}]+[\psi_{H}^{1,0},\psi_{H}^{0,1}])=\lambda Id_{E},
\end{split}\right.
\end{equation}
where we used the NHYM condition.

For any two Hermitian metrics $K$ and $H$ on $E$, then we have
\begin{equation}
    \begin{split}
    &D_H=h^{-1}\circ D_K \circ h+\frac{1}{2}(D-h^{-1}\circ D\circ h),\\
    &\psi_{H}=h^{-1}\circ\psi_{K}\circ h+\frac{1}{2}(D-h^{-1}\circ D\circ h),
    \end{split}
\end{equation}
where $h=K^{-1}H$ is a positive endomorphism which is defined by $H(\cdot,\cdot)=K(h(\cdot),\cdot)$. If we set $D_{K}^{c}=D^{''}_{K}-D_{K}^{'}$, then a straightforward calculation shows that
\begin{equation}\label{eq:n1}
\begin{split}
    D_{H}^{''}-D^{''}_{K}=&\frac{1}{2}h^{-1}D_{K}^{c}h,\\
    \sqrt{-1}\Lambda_{\omega}G_{H}-\sqrt{-1}\Lambda_{\omega}G_{K}=&\frac{\sqrt{-1}}{4}\Lambda_{\omega}D(h^{-1}D_{K}^{c}h),
\end{split}
\end{equation}
where we used the equation (\ref{eq:111}).

\subsection{Harmonic metric on NHYM bundles}
In the following, we show that the existence of harmonic metrics implies the semisimplicity of $E$. Let $S\subset E$ be a $D$-invariant subbundle, $Q$ be the quotient bundle. Then any Hermitian metric $H$ on $E$ induces a smooth isomorphism $f_{H}:S\oplus Q\rightarrow E$. Let
\begin{equation}
    f^{*}_{H}(D)=\left(\begin{split}
    &D_{S}\ &2\beta\\
    &0\ &D_{Q}
    \end{split}\right),
\end{equation}
where $\beta\in \Omega^{1}(X,Q^{*}\otimes S)$. Then
\begin{equation}
    f^{*}_{H}(D_{H})=\left(\begin{split}
    &D_{S,H_{S}}\ &\beta\\
    &-\beta^{*}\ &D_{Q,H_{Q}}
    \end{split}\right),\ \ \ \ f^{*}_{H}(\psi_{H})=\left(\begin{split}
    &\psi_{S,H_{S}}\ &\beta\\
    &\beta^{*}\ &\psi_{Q,H_{Q}}
    \end{split}\right),
\end{equation}
and
\begin{equation}
    f^{*}_{H}(G_{H})=\left(\begin{split}
    &G_{S,H_{S}}+\beta\wedge(\beta^{c})^{*}\ & D^{''}\beta\\
    &D^{''}(\beta^{c})^{*}\ &G_{Q,H_{Q}}+(\beta^{c})^{*}\wedge\beta^{c}
    \end{split}\right),
\end{equation}
where $\beta^{c}=\beta^{0,1}-\beta^{1,0}$. Taking the trace and the integral on both sides, we see
\begin{equation}
\int_{X}(\sqrt{-1}\Lambda_{\omega}\tr(G_{S,H_{S}})-|\beta|^{2})\frac{\omega^{n}}{n!}=0.
\end{equation}
On the other hand, by the Stokes' formula, we have
\begin{equation}
\int_{X}\sqrt{-1}\Lambda_{\omega}\tr(G_{S,H_{S}})\frac{\omega^{n}}{n!}=0.
\end{equation}
Therefore, we have
\begin{equation}\label{eee}
\beta=0,
\end{equation}
which means that $(E,D)\simeq (S,D_{S})\oplus(Q,D_{Q})$. Thus  we obtain that $(E,D)$ is semisimple by induction on the rank.

\section{Donaldson's functional on NHYM bundles}\label{sec:DF}
In this section, we wil introduce the Donaldson's functional by following Donaldson (\cite{Don2}) and Simpson (\cite{S1}) in our situation.
For any two Hermitian metrics $K,L$, we now define the Donaldson's functional on NHYM bundles as follows
\begin{equation}
  \mathcal{M}(K,L)=\int_{0}^1\int_X{\rm tr}\bigg(-4\sqrt{-1}\Lambda_{\omega}G_{H(s)}\cdot H^{-1}(s)\frac{\partial H(s)}{\partial s}\bigg)\frac{\omega^n}{n!}ds,
\end{equation}
where $H(s)$ is a path connecting the Hermitian metrics $K$ and $L$ in the space of all metrics on $E$.

\begin{prop}\label{DFP1}
  The Donaldson's functional $\mathcal{M}$ is independent of the path.
\end{prop}
\begin{proof}
  Let $H(\tau,s),s\in[0,1]$ be a family of path with $\tau\in (-\varepsilon,\varepsilon)$ such that $H(\tau,0)=K, H(\tau,1)=L$. If we denote $h(\tau,s)=K^{-1}H(\tau,s)$, then $h(\tau,0)={\rm Id}, h(\tau,1)=K^{-1}L$. For writing convenience, we omit the parameters here. By direct calculation, we have
  \begin{equation}\label{DFC1}
    \frac{\partial }{\partial \tau}\bigg(\sqrt{-1}\Lambda_{\omega}G_H\bigg)=\frac{\sqrt{-1}}{4}\Lambda_{\omega}DD_H^c\bigg(h^{-1}\frac{\partial h}{\partial \tau}\bigg).
  \end{equation}
  Since $D$ is a NHYM connection, we also have
  \begin{equation}\label{DFC2}
    \sqrt{-1}\Lambda_{\omega}G_H=\frac{\sqrt{-1}}{4}\Lambda_{\omega}[D,D_{H}^c].
  \end{equation}
  By using (\ref{DFC1}) and (\ref{DFC2}), we can easily obtain that
  \begin{equation}\label{DFC3}
    \begin{split}
      &\quad\frac{\partial}{\partial \tau}{\rm tr}\bigg(\sqrt{-1}\Lambda_{\omega}G_H \cdot h^{-1}\frac{\partial h}{\partial s}\bigg)-\frac{\partial}{\partial s}{\rm tr}\bigg(\sqrt{-1}\Lambda_{\omega}G_H \cdot h^{-1}\frac{\partial h}{\partial \tau}\bigg)\\
      &={\rm tr}\bigg(\frac{\sqrt{-1}}{4}\Lambda_{\omega}DD_H^c\big(h^{-1}\frac{\partial h}{\partial \tau}\big)h^{-1}\frac{\partial h}{\partial s}-\frac{\sqrt{-1}}{4}\Lambda_{\omega}DD_H^c\big(h^{-1}\frac{\partial h}{\partial s}\big)h^{-1}\frac{\partial h}{\partial \tau}\bigg)\\
      &\quad +{\rm tr}\bigg(-\sqrt{-1}\Lambda_{\omega}G_H\cdot h^{-1}\frac{\partial h}{\partial \tau}h^{-1}\frac{\partial h}{\partial s}+\sqrt{-1}\Lambda_{\omega}G_H\cdot h^{-1}\frac{\partial h}{\partial s}h^{-1}\frac{\partial h}{\partial \tau}\bigg)\\
      &={\rm tr}\bigg(\frac{\sqrt{-1}}{4}\Lambda_{\omega}DD_H^c\big(h^{-1}\frac{\partial h}{\partial \tau}\big)h^{-1}\frac{\partial h}{\partial s}-\frac{\sqrt{-1}}{4}\Lambda_{\omega}DD_H^c\big(h^{-1}\frac{\partial h}{\partial s}\big)h^{-1}\frac{\partial h}{\partial \tau}\bigg)\\
      &\quad +{\rm tr}\bigg(\frac{\sqrt{-1}}{4}\Lambda_{\omega}[D,D^c_H]\big(h^{-1}\frac{\partial h}{\partial s}\big)h^{-1}\frac{\partial h}{\partial \tau}\bigg)\\
      &=\frac{\sqrt{-1}}{4}\Lambda_{\omega}{\rm tr}\bigg(DD_H^c\big(h^{-1}\frac{\partial h}{\partial \tau}\big)h^{-1}\frac{\partial h}{\partial s}+D_H^cD\big(h^{-1}\frac{\partial h}{\partial s}\big)h^{-1}\frac{\partial h}{\partial \tau}\bigg)\\
      &=\frac{\sqrt{-1}}{4}\Lambda_{\omega}d{\rm tr}\bigg(D^c_H\big(h^{-1}\frac{\partial h}{\partial \tau}\big)h^{-1}\frac{\partial h}{\partial s}\bigg)+\frac{\sqrt{-1}}{4}\Lambda_{\omega}d^c{\rm tr}\bigg(D\big(h^{-1}\frac{\partial h}{\partial s}\big)h^{-1}\frac{\partial h}{\partial \tau}\bigg).
    \end{split}
  \end{equation}
  Using the Stokes' formula together with (\ref{DFC3}) yields that
  \begin{equation}
    \int_X\frac{\partial}{\partial \tau}{\rm tr}\bigg(\sqrt{-1}\Lambda_{\omega}G_H \cdot h^{-1}\frac{\partial h}{\partial s}\bigg)-\frac{\partial}{\partial s}{\rm tr}\bigg(\sqrt{-1}\Lambda_{\omega}G_H \cdot h^{-1}\frac{\partial h}{\partial \tau}\bigg)\frac{\omega^n}{n!}=0.
  \end{equation}
  Hence,
  \begin{equation*}
    \begin{split}
      &\quad \frac{d}{d\tau}\int_0^1\int_X{\rm tr}\bigg(-4\sqrt{-1}\Lambda_{\omega}G_{H(\tau,s)}\cdot H^{-1}(\tau,s)\frac{\partial H(\tau,s)}{\partial s}\bigg)\frac{\omega^n}{n!}ds\\
      &=\int_0^{1}\int_X\frac{\partial}{\partial \tau}\tr\bigg(-4\sqrt{-1}\Lambda_{\omega}G_{H(\tau,s)}\cdot H^{-1}(\tau,s)\frac{\partial H(\tau,s)}{\partial s}\bigg)\frac{\omega^n}{n!}ds\\
      &=\int_0^{1}\int_X\frac{\partial}{\partial s}\tr\bigg(-4\sqrt{-1}\Lambda_{\omega}G_{H(\tau,s)}\cdot H^{-1}(\tau,s)\frac{\partial H(\tau,s)}{\partial \tau}\bigg)\frac{\omega^n}{n!}ds\\
      &=\int_X{\rm tr}\bigg(-4\sqrt{-1}\Lambda_{\omega}G_{H(\tau,s)}\cdot H^{-1}(\tau,s)\frac{\partial H(\tau,s)}{\partial \tau}\bigg)\frac{\omega^n}{n!}\bigg|_{s=0}^1\\
      &=0.
    \end{split}
  \end{equation*}
  This means that the Donaldson's functional is independent of the path.
\end{proof}

Let $s$ be the endomorphism of $E$ determined by the condition $L=Ke^s$. Choosing a special path $H(t)=Ke^{ts}$ connecting $K$ and $L$, then the Donaldson's functional can be written as follows
\begin{equation}
  \mathcal{M}(K,L)=\int_0^1\int_{X}{\rm tr}\bigg(\big(-4\sqrt{-1}\Lambda_{\omega}G_{H(t)}\big)s\bigg)\frac{\omega^n}{n!}dt.
\end{equation}
Set
\begin{equation*}
  \Psi(x,y):=
  \frac{e^{y-x}-(y-x)-1}{(y-x)^2}.
\end{equation*}
Then we have the following Proposition.
\begin{prop}Let $L=Ke^s$, then we have
\begin{equation}\label{DFkey}
  \mathcal{M}(K,L)=\int_{X}{\rm tr}\bigg(\big(-4\sqrt{-1}\Lambda_{\omega}G_{K}\big)s\bigg)\frac{\omega^n}{n!}+\int_{X}\langle\Psi(s)(Ds),Ds\rangle_{K}
  \frac{\omega^n}{n!}.
\end{equation}
\end{prop}
\begin{proof}
  By directly calculation, we deduce
  \begin{equation}
    \frac{\partial}{\partial t}\mathcal{M}(K,H(t))=\int_{X}{\rm tr}\bigg(\big(-4\sqrt{-1}\Lambda_{\omega}G_{H(t)}\big)s\bigg)\frac{\omega^n}{n!}
  \end{equation}
  and
  \begin{equation}\label{DFC4}
  \begin{split}
    &\quad \frac{\partial^2}{\partial t^2}\mathcal{M}(K,H(t))=\int_X{\rm tr}\bigg(\big(-\sqrt{-1}\Lambda_{\omega}DD_{H(t)}^cs\big)s\bigg)\frac{\omega^{n}}{n!}\\
    &=\int_{X}-\sqrt{-1}\Lambda_{\omega}d{\rm tr}(D_{H(t)}^cs\cdot s)\frac{\omega^{n}}{n!}-\int_{X}\sqrt{-1}\Lambda_{\omega}{\rm tr}\big(D_{H(t)}^cs\wedge Ds\big)\frac{\omega^n}{n!}\\
    &=\int_{X}|Ds|_{H(t),\omega}^2\frac{\omega^n}{n!},
    \end{split}
  \end{equation}
  where we used $\sqrt{-1}\Lambda_{\omega}{\rm tr}(D_H^cs \wedge Ds)=|Ds|_{H,\omega}^2$. Choosing a orthonormal basis $\{e_1,\cdots,e_r\}$ with respect to the metric $K$, such that $s=\sum_{i=1}^{r}\lambda_{i}e_{i}\otimes e_{i}^{*}$. Let $Ds=\sum_{i,j=1}^{r}(Ds)_{i}^{j}e_{j}\otimes e_{i}^{*}$, then we have
  \begin{equation}\label{DFC5}
    |Ds|_{H(t),\omega}^2=\sum_{i,j=1}^{r}|(Ds)_i^j|_{\omega}^2e^{t(\lambda_j-\lambda_i)}.
  \end{equation}
  Using (\ref{DFC5}) and integrating (\ref{DFC4}) from $0$ to $t$ yields
  \begin{equation}
    \frac{\partial}{\partial t}\mathcal{M}(K,H(t))=\sum_{i,j=1}^{r}\int_X|(Ds)_i^j|_{\omega}^2\frac{e^{t(\lambda_j-\lambda_i)}-1}{(\lambda_j-\lambda_i)}
    \frac{\omega^n}{n!}+\int_X{\rm tr}\bigg
    (\big(-4\sqrt{-1}\Lambda_{\omega}G_{K}\big)s\bigg)\frac{\omega^n}{n!}.
  \end{equation}
  Hence,
  \begin{equation*}
    \begin{split}
      &\quad \mathcal{M}(K,L)=\int_{0}^1\frac{\partial}{\partial t}\mathcal{M}(K,H(t))dt\\
      &=\int_{X}{\rm tr}\bigg(\big(-4\sqrt{-1}\Lambda_{\omega}G_{K}\big)s\bigg)\frac{\omega^n}{n!}+\sum_{i,j=1}^{r}\int_X|(Ds)_i^j|_{\omega}^2
      \Psi(\lambda_{i},\lambda_{j})\frac{\omega^n}{n!}\\
      &=\int_{X}{\rm tr}\bigg(\big(-4\sqrt{-1}\Lambda_{\omega}G_{K}\big)s\bigg)\frac{\omega^n}{n!}+\int_{X}\langle\Psi(s)(Ds),Ds\rangle_{K}
  \frac{\omega^n}{n!}.
    \end{split}
  \end{equation*}
\end{proof}

The following property are easily derived from the definition of  the Donaldson's functional.
\begin{prop}
	Let $K$, $J$ and $L$ be three Hermitian metrics on $(E,D)$, then
	\begin{equation}
		\mathcal{M}(K,J)+\mathcal{M}(J,L)=\mathcal{M}(K,L).
	\end{equation}
\end{prop}

\section{The heat flow on NHYM bundle}\label{sec:FL}
In this section, we consider the following heat flow on the NHYM bundle $(E,D)$ over compact K\"ahler manifolds:
\begin{equation}\label{HFeq1}
    H^{-1}(t)\frac{\partial H(t)}{\partial t}=4\sqrt{-1}\Lambda_{\omega}G_{H(t)}.
\end{equation}
For convenience, we set $\Phi(H)=4\sqrt{-1}\Lambda_{\omega}G_{H}$. It is easy to see  that $\Phi(H)$ is self-adjoint with respect to $H$. Next we will give some basic properties about this flow.
\begin{prop}\label{HFP1}
  Let $H(t)$ be a solution  of the heat flow (\ref{HFeq1}) with initial metric $H_{0}$, then
  \begin{equation}\label{HFC1}
    \bigg(\frac{\partial }{\partial t}-\Delta\bigg){\rm tr}\big(\Phi(H(t))\big)=0,
  \end{equation}
  and
  \begin{equation}\label{HFC2}
    \bigg(\frac{\partial }{\partial t}-\Delta\bigg)|\Phi(H(t))|_{H(t)}^{2}=-2|D\Phi(H(t))|_{H(t),\omega}^2.
  \end{equation}
\end{prop}
\begin{proof}
  Set $h(t)=H_0^{-1}H(t)$. We omit the parameter $t$ in the computations when there is no confusion.  Making use of identity (\ref{eq:n1}),
  we have
  \begin{equation}\label{HFC3}
  \begin{split}
    \frac{\partial}{\partial t}\Phi(H)&=\frac{\sqrt{-1}}{4}\Lambda_{\omega}D\bigg(-h^{-1}\frac{\partial h}{\partial t}h^{-1}D_{H_0}^ch+h^{-1}D_{H_0}^c\frac{\partial h}{\partial t}\bigg)\\
    &=\frac{\sqrt{-1}}{4}\Lambda_{\omega}D\bigg(-h^{-1}\frac{\partial h}{\partial t}h^{-1}D_{H_0}^ch+h^{-1}D_{H_0}^ch h^{-1}\frac{\partial h}{\partial t}+D_{H_0}^c\big(h^{-1}\frac{\partial h}{\partial t}\big)\bigg)\\
    &=\frac{\sqrt{-1}}{4}\Lambda_{\omega}D\bigg([h^{-1}D_{H_0}^ch,h^{-1}\frac{\partial h}{\partial t}]+D_{H_0}^c\big(h^{-1}\frac{\partial h}{\partial t}\big)\bigg)\\
    &=\sqrt{-1}\Lambda_{\omega}DD_{H}^c\Phi(H).
    \end{split}
  \end{equation}
  The formula (\ref{HFC1}) can be easily deduced form (\ref{HFC3}). On the other hand, we have
  \begin{equation}\label{HFC4}
  \begin{split}
    \Delta|\Phi(H)|_{H}^2&=\sqrt{-1}\Lambda_{\omega}dd^c|\Phi(H)|_{H}^2=\sqrt{-1}\Lambda_{\omega}dd^c{\rm tr}(\Phi(H)\Phi(H))\\
    &=\sqrt{-1}\Lambda_{\omega}d\{{\rm tr} (D_H^c\Phi(H)\cdot \Phi(H)+\Phi(H)\cdot D_H^c\Phi(H))\}\\
    &=\sqrt{-1}\Lambda_{\omega}{\rm tr}\left\{DD_H^c\Phi(H)\cdot\Phi(H)-D_H^c\Phi(H)\wedge D\Phi(H)\right.\\
    &\quad +\left. D\Phi(H)\wedge D_{H}^c\Phi(H)+\Phi(H)DD_{H}^c\Phi(H)\right\}\\
    &=2\sqrt{-1}\Lambda_{\omega}{\rm tr}(\Phi(H)DD_H^c\Phi(H))+2|D\Phi(H)|_H^2
    \end{split}
  \end{equation}
  Combining (\ref{HFC3}) and (\ref{HFC4}) yields
  \begin{equation*}
    \begin{split}
      \bigg(\frac{\partial}{\partial t}-\Delta\bigg)|\Phi(H)|_H^2&=2{\rm tr}(\frac{\partial}{\partial t}\Phi(H)\Phi(H))-\Delta|\Phi(H)|_H^2\\
      &=-2|D\Phi(H)|_H^2.
    \end{split}
  \end{equation*}
\end{proof}
By Proposition \ref{HFP1} and the maximum principle, we can easily deduce that
\begin{equation}\label{eq:n2}
	\sup_{X}|\Phi(H(t))|_{H(t)}\leq \sup_{X}|\Phi(H_{0})|_{H_{0}}.
\end{equation}
If we take the initial metric $H_0$ such that $\tr(\Phi(H_{0}))=0$, then $\tr(\Phi(H(t)))=0$. Combining with the equation (\ref{HFeq1}), it can be derived that
\begin{equation}
\log\det (H_{0}^{-1}H(t))=\tr(\log(H_{0}^{-1}H(t)))=0.
\end{equation}

Following Donaldson's argument (\cite{Don2}), we introduce the Donaldson's distance on the space of the Hermitian metrics as follows.
\begin{defn}
  Given any two Hermitian metrics $K$ and $H$ on the vector bundle $E$, the Donaldson's distance is defined by
  \begin{equation}
    \sigma(K,H)={\rm tr}(K^{-1}H)+{\rm tr}(H^{-1}K)-2{\rm rank }E.
  \end{equation}
\end{defn}
It is obvious that $\sigma(K,H)\geq 0$ with the equality holds if and only if $K=H$. And  a sequence of Hermitian metrics $H_i$ converge to some $H$ in the usual $C^0$-topology if and only if
\begin{equation*}
  \lim_{i\to \infty}\sup_{X}\sigma(H_i,H)=0.
\end{equation*}
\begin{prop}\label{HFP2}
  Let $H$ and $K$ be two harmonic metrics on the NHYM bundle $(E,D)$, then we have
  \begin{equation}
    \Delta\sigma(H,K)=|h^{-\frac{1}{2}}\cdot Dh|_K^2+|h^{\frac{1}{2}}\cdot Dh^{-1}|_H^2\geq 0.
  \end{equation}
where $h=K^{-1}H$.
\end{prop}
\begin{proof}
  By making use of (\ref{eq:n1}), we obtain
  \begin{equation}
    {\rm tr}\bigg(4h(\sqrt{-1}\Lambda_{\omega}G_H-\sqrt{-1}\Lambda_{\omega}G_K)\bigg)=\Delta {\rm tr}h-\sqrt{-1}\Lambda_{\omega}{\rm tr}(Dh\wedge h^{-1}D_K^ch),
  \end{equation}
  and
  \begin{equation}
  {\rm tr}\bigg(4h^{-1}(\sqrt{-1}\Lambda_{\omega}G_K-\sqrt{-1}\Lambda_{\omega}G_H)\bigg)=\Delta {\rm tr}h^{-1}-\sqrt{-1}\Lambda_{\omega}{\rm tr}(Dh^{-1}\wedge hD_K^ch^{-1}).
  \end{equation}
  Hence, we complete the proof by the following identities
  \begin{equation}
    \sqrt{-1}\Lambda_{\omega}{\rm tr}(Dh\wedge h^{-1}D_K^ch)=|h^{-\frac{1}{2}}\cdot Dh|_K^2,
  \end{equation}
  and
  \begin{equation}
    \sqrt{-1}\Lambda_{\omega}{\rm tr}(Dh^{-1}\wedge hD_{H}^ch^{-1})=|h^{\frac{1}{2}}\cdot Dh^{-1}|_H^2.
  \end{equation}
\end{proof}
By Proposition \ref{HFP2} and \ref{p:un}, we have the following corollary.
\begin{cor}
	Let $H_{1},H_{2}$ be two harmonic metrics, then $H_{1}=e^{c}H_{2}$ where $c\in \mathbb{R}$.
\end{cor}

If $H(t)$, $K(t)$ are two solutions of the heat flow (\ref{HFeq1}), then using the same method as Proposition \ref{HFP2}, we can easily deduce the following proposition.
\begin{prop}\label{HFP3}
  Let $H(t)$, $K(t)$ be two solutions of the heat flow (\ref{HFeq1}), then we have
  \begin{equation}
    \bigg(\frac{\partial }{\partial t}-\Delta\bigg)\sigma(H(t),K(t))=-|h^{-\frac{1}{2}}(t)\cdot Dh(t)|_{K(t)}^2-|h^{\frac{1}{2}}(t)\cdot Dh^{-1}(t)|_{H(t)}^2,
  \end{equation}
  where $h(t)=K^{-1}(t)H(t)$.
\end{prop}

\begin{prop}\label{Prop:Ineq}
  Let $H$ and $K$ are two Hermitian  metrics on the NHYM bundle $(E,D)$, then
  \begin{equation}\label{p:eq1}
    \Delta \log({\rm tr}h+{\rm tr}h^{-1})\geq -4|\sqrt{-1}\Lambda_{\omega}G_H|_H-4|\sqrt{-1}\Lambda_{\omega}G_K|_K,
  \end{equation}
  where $h=K^{-1}H$.  Moreover
  \begin{equation}\label{p:eq2}
  \sup_{X}|\log(h)|_{K}\leq C_{1}\|\log(h)\|_{L^{1}(X,K)}+C_2,
 \end{equation}
 where $C_{1}$ and $C_{2}$ are positive constants depending only on $\sup_{X}|\sqrt{-1}\Lambda_{\omega}G_{H}|_{H}$, $\sup_{X}|\sqrt{-1}\Lambda_{\omega}G_{K}|_{K}$ and the geometry of $(X,\omega)$.
\end{prop}
\begin{proof}
  By making  use of (\ref{eq:n1}), we obtain
  \begin{equation}\label{HFC5}
    \Delta {\rm tr}h\geq |h^{-\frac{1}{2}}\cdot Dh|_K^2-4{\rm tr}h \cdot( |\sqrt{-1}\Lambda_{\omega}G_{H}|_{H}+|\sqrt{-1}\Lambda_{\omega}G_{K}|_{K}),
  \end{equation}
  and
  \begin{equation}\label{HFC6}
    \Delta {\rm tr}h^{-1}\geq |h^{\frac{1}{2}}\cdot Dh^{-1}|_H^2-4{\rm tr}h^{-1}\cdot (|\sqrt{-1}\Lambda_{\omega}G_{H}|_{H}+|\sqrt{-1}\Lambda_{\omega}G_K|_K).
  \end{equation}
  Combining (\ref{HFC5}) and (\ref{HFC6}) yields
  \begin{equation}\label{HFC7}
  \begin{split}
    \Delta({\rm tr}h+{\rm tr}h^{-1})&\geq |h^{-\frac{1}{2}}\cdot Dh|_K^2+|h^{\frac{1}{2}}\cdot Dh^{-1}|_H^2\\
    &\quad -4({\rm tr}h+{\rm tr}h^{-1})\big(|\sqrt{-1}\Lambda_{\omega}G_{H}|_{H}+|\sqrt{-1}\Lambda_{\omega}G_K|_K\big).
    \end{split}
  \end{equation}
  On the other hand, note that
  \begin{equation}\label{HFC8}
    \Delta\log({\rm tr}h+{\rm tr}h^{-1})=-\frac{|{\rm tr}(Dh)+{\rm tr}(Dh^{-1})|^2}{({\rm tr}h+{\rm tr}h^{-1})^2}+\frac{\Delta({\rm tr}h+{\rm tr}h^{-1})}{{\rm tr}h+{\rm tr}h^{-1}}.
  \end{equation}
  Combining (\ref{HFC7}) and (\ref{HFC8}), it suffices to prove
  \begin{equation*}
    \frac{|{\rm tr}(Dh)+{\rm tr}(Dh^{-1})|^2}{{\rm tr}h+{\rm tr}h^{-1}}\leq |h^{-\frac{1}{2}}\cdot Dh|_K^2+|h^{\frac{1}{2}}\cdot Dh^{-1}|_H^2,
  \end{equation*}
  which is valid from the Young's inequality immediately. So we obtain (\ref{p:eq1}).  Noting
\begin{equation}
\log\big(\frac{1}{2r}(\tr(h)+\tr(h^{-1}))\big)\leq |\log(h)|_{K}\leq r^{1/2}\log(\tr(h)+\tr(h^{-1})).
\end{equation}
Applying Moser's iteration, we finish the proof.
\end{proof}

\subsection{Long-time existence of the heat flow}
In this subsection, we consider the long-time existence of the solution to the heat flow (\ref{HFeq1}). Let $H(t)$ be a solution of (\ref{HFeq1}) with initial metric $H_{0}$. Set $h(t)=H_{0}^{-1}H(t)$, then (\ref{HFeq1}) can be written as:
\begin{equation}\label{HFeq2}
    \frac{\partial h(t)}{\partial t}=h(t)(\sqrt{-1}\Lambda_{\omega}D(h^{-1}(t)D_{H_{0}}^{c}h(t))+4\sqrt{-1}\Lambda_{\omega}G_{H_{0}}),
\end{equation}
By above formula, it is easy to see that the heat flow (\ref{HFeq2}) is strictly parabolic, so we obtain the short-time existence from the standard PDE theory.
\begin{thm}\label{ThmHF1}
  For sufficiently small $T>0$, the equation (\ref{HFeq1}) has a smooth solution $H(t)$ defined for $0\leq t<T$.
\end{thm}

In the following, we prove the long-time existence of (\ref{HFeq1}) following the standard argument in \cite{Don2}.

\begin{lem}\label{HFL1}
  Let $H(t)$ be a smooth solution of the equation (\ref{HFeq1}) defined for $0\leq t<T<\infty$. Then $H(t)$ converges in $C^0$-topology to some continuous non-degenerate Hermitian metric $H_T$ as $t\to T$.
\end{lem}
\begin{proof}
  Given $\varepsilon>0$, by the continuity at $t=0$, we can find $\delta>0$ such that
  \begin{equation*}
    \sup_{X}\sigma(H(t),H(t'))<\varepsilon, \quad \text{for} \; 0<t,t'<\delta.
  \end{equation*}
  Combining Proposition \ref{HFP3} and the maximum principle yields that
  \begin{equation*}
    \sup_{X}\sigma(H(t),H(t'))<\varepsilon, \quad \text{for} \; T-\delta<t,t'<T.
  \end{equation*}
  This means that $H(t)$ is a  Cauchy sequence and converges to a continuous metric $H_T$.

  So, it remains to show that $H_T$ is a non-degenerate metric. By direct calculation, we have
  \begin{equation}\label{HFC9}
    \bigg|\frac{\partial}{\partial t}\log{\rm tr}h(t)\bigg|\leq |\Phi(H(t))|_{H(t)},
  \end{equation}
  and
  \begin{equation}\label{HFC10}
    \bigg|\frac{\partial}{\partial t}\log{\rm tr}h^{-1}(t)\bigg|\leq |\Phi(H(t))|_{H(t)}.
  \end{equation}
  Hence, (\ref{HFC9}) and (\ref{HFC10}) together with Proposition \ref{HFP1} imply that  $\sigma(H(t),H_0)$  is uniformly bounded on $X\times [0,T)$. Therefore, $H_T$ is a non-degenerate metric.
\end{proof}
Using the same method as in \cite{Don2}, we get the following lemma.
\begin{lem}\label{HFL2}
  Let $H(t),0\leq t< T$ be a  family of Hermitian metrics on a non-Hermitian Yang-Mills bundle $E$ over $X$. Suppose $H(t)$ converges in $C^0$-topology to some continuous metric $H_T$ as $t\to T$ and $\sup_{X}|\sqrt{-1}\Lambda_{\omega}G_{H(t)}|_{H(t)}$ is uniformly bounded in $t$, then $H(t)$ are bounded in $C^1$ and also bounded in $L_2^p$ for any $1<p<\infty$.
\end{lem}

\begin{thm}
  The equation (\ref{HFeq1}) has a unique long-time solution $H(t)$ for $0\leq t<+\infty$.
\end{thm}
\begin{proof}
  By Theorem \ref{ThmHF1}, we know that the solution $H(t)$ of the equation (\ref{HFeq1}) exists  for a short time defined for $0\leq t<T$. On the one hand,  we deduce that $H(t)$ converges in $C^0$-topology to a non-degenerate  continuous metric $H(T)$ as $t\to T$ from Lemma \ref{HFL1}. On the other hand, using Proposition (\ref{HFP1}) and maximum principle implies that $\sup_{X}|\Lambda_{\omega}G_{H(t)}|_{H_0}$  is uniformly bounded in $t$. Hence, by Lemma \ref{HFL2}, $H(t)$ is bounded in $C^1$ and also bounded in $L_2^p$ for $1<p<\infty$. Finally, we can use the Hamilton's method (\cite{H}) to deduce that $H(t)$ converges to $H(T)$ in $C^{\infty}$-topology and the solution can be continued past $T$. Therefore, the equation (\ref{HFeq1}) has a solution $H(t)$ which  exists for all time.

  It is easy to obtain the uniqueness of the solution from Proposition \ref{HFP3} and the maximum principle.
\end{proof}

In the following, we will derive the $C^{1}$-estimates of the solution using maximum principle. Let $H(t)$ be the solution of equation (\ref{HFeq1}) with initial metric $H_{0}$. By direct calculation, we have
\begin{equation}
	\frac{\partial}{\partial t}(\sigma(H_{0},H(t))+2r)\leq 2|\Phi(H(t))|_{H(t)}(\sigma(H_{0},H(t))+2r).
\end{equation}
Then we obtain the following $C^{0}$-estimate of $H(t)$.
\begin{prop}
	Let $C=\sup_{X}\Phi(H_{0})$, then
	\begin{equation}
		\sup_{X}\sigma(H_{0},H(t))\leq 2r(e^{Ct}-1).
	\end{equation}
\end{prop}

If we denote the anti-self-adjoint part (resp. self-adjoint part) of $F_D$ by $F_D^{+}$ (resp. $F_D^{-}$), i.e., $F_{D}^{+}=D_{H}^{2}+\psi_{H}\wedge\psi_{H}$, $F_{D}^{-}=D_{H}\psi_{H}$, then we have
\begin{equation}
	F_{D}=F_{D}^{+}+F_{D}^{-},\ \ \ \ |F_{D}|^{2}_{H,\omega}=|F_{D}^{+}|^{2}_{H,\omega}+|F_{D}^{-}|^{2}_{H,\omega}.
\end{equation}
By direct calculation, we obtain the following Bochner-type formula
\begin{equation}
	\begin{split}
		-\Delta|\psi_{H}|^{2}_{H,\omega}=&-2\langle\Box^{+}\psi_{H},\psi_{H}\rangle_{H,\omega}-2|\nabla\psi_{H}|^{2}_{H,\omega}-2\langle \psi_{H}\circ Ric,\psi_{H}\rangle_{H,\omega}\\
		&-2|\psi_{H}\wedge\psi_{H}|^{2}_{H,\omega}+4\langle F_{D}^{+},\psi_{H}\wedge\psi_{H}\rangle,
	\end{split}
\end{equation}
where $\Box^{+}=-D_{H}D_{H}^{*}-D_{H}^{*}D_{H}$ and $\psi_{H}\circ Ric=g^{jk}Ric_{ij}\psi_{k}dx^{i}$. So we can easily deduce that
\begin{equation}
	\begin{split}
		\bigg(\frac{\partial}{\partial t}-\Delta\bigg)|\psi_{H}|^{2}_{H,\omega}=&-2|\nabla\psi_{H}|^{2}_{H,\omega}-2\langle \psi_{H}\circ Ric,\psi_{H}\rangle_{H,\omega}\\
		&-2|\psi_{H}\wedge\psi_{H}|^{2}_{H,\omega}+4\langle F_{D}^{+},\psi_{H}\wedge\psi_{H}\rangle_{H,\omega}+2\langle D_{H}^{*}D_{H}\psi_{H},\psi_{H}\rangle_{H,\omega}.
	\end{split}
\end{equation}
Since $\sqrt{-1}\Lambda_{\omega}D_{H}\psi_{H}=0$, by the K\"ahler identities, we have
\begin{equation}
	D_{H}^{*}D_{H}\psi_{H}=\sqrt{-1}\Lambda_{\omega}D_{H}^{c}D_{H}\psi_{H}.
\end{equation}
For simplicity, we denote the $(1,0)$-part of $\psi_H$ by $\theta$, then
\begin{equation}
	\begin{split}
		D_{H}^{c}D_{H}\psi_{H}=&-[\theta^{*}\wedge\theta^{*},\theta]+[\theta\wedge\theta,\theta^{*}]+[F_{D}^{+},\theta^{*}]\\
		&-[[\theta,\theta^{*}],\theta^{*}]-[F_{D}^{+},\theta]+[[\theta,\theta^{*}],\theta].
	\end{split}
\end{equation}
A direct calculation shows that
\begin{equation}
	\begin{split}
		\langle \sqrt{-1}\Lambda_{\omega}D_{H}^{c}D_{H}\psi_{H},\psi_{H}\rangle_{H}=&|[\theta,\theta]|^{2}_{H,\omega}+|[\theta,\theta]|^{2}_{H,\omega}+\langle F^{+}_{D},[\theta,\theta^{*}]\rangle_{H,\omega}\\
		&-|[\theta,\theta]|^{2}_{H,\omega}+\langle F^{+}_{D},[\theta,\theta^{*}]\rangle_{H,\omega}-|[\theta,\theta]|^{2}_{H,\omega}\\
		=&2\langle F^{+}_{D},\psi_{H}\wedge\psi_{H}\rangle_{H,\omega}.
	\end{split}
\end{equation}
Combining above inequalities, we have the following proposition.
\begin{prop}\label{p:c1} Let $(E,D)$ be a NHYM bundle over the compact K\"ahler manifold $(X,\omega)$. Given any Hermitian metirc $H$, we have
\begin{equation}
	\begin{split}
		\bigg(\frac{\partial}{\partial t}-\Delta\bigg)|\psi_{H}|^{2}_{H,\omega}=&-2|\nabla\psi_{H}|^{2}_{H,\omega}-2\langle \psi_{H}\circ Ric,\psi_{H}\rangle_{H,\omega}\\
		&-2|\psi_{H,D}\wedge\psi_{H,D}|^{2}_{H,\omega}+8\langle F_{D}^{+},\psi_{H}\wedge\psi_{H}\rangle_{H,\omega}.
	\end{split}
\end{equation}
\end{prop}

\begin{prop}\label{p:last}
  Let $T\in(0,+\infty]$. Assume that there exists a constant $C_{1}>0$ such that
  \begin{equation}
      \sup_{X\times[0,T)}|\log(h)|_{K}<C_{1},
  \end{equation}
  then there exists a constant $C_{2}>0$ depending only on $C_{1}$ and the geometry of $(X,\omega)$ such that
  \begin{equation}
  	\sup_{X\times[0,T)}|\psi_{H}|_{H,\omega}<C_{2}.
  \end{equation}
\end{prop}
\begin{proof}
Computing directly, we obtain
	\begin{equation}
		\bigg(\frac{\partial}{\partial t}-\Delta\bigg)\tr(h)=-|Dh\cdot h^{-1/2}|^{2}_{H,\omega}+\tr(h\cdot\Phi(H_{0})).
	\end{equation}
Since $|Dh\cdot h^{-1/2}|^{2}_{H,\omega}\geq \tilde{C}_{1}|\psi_{H}|_{H,\omega}^{2}-\tilde{C}_{2}$, we have
	\begin{equation}
	\bigg(\frac{\partial}{\partial t}-\Delta\bigg)\tr(h)\leq -\tilde{C}_{1}|\psi_{H}|_{H,\omega}^{2}+\tilde{C}_{3}.
\end{equation}
By Proposition \ref{p:c1}, we have
\begin{equation}
	\begin{split}
		\bigg(\frac{\partial}{\partial t}-\Delta\bigg)|\psi_{H}|^{2}_{H,\omega}\leq &\tilde{C}_{4}|\psi_{H}|_{H,\omega}^{2}.
	\end{split}
\end{equation}
Consider the following test function
\begin{equation}
	f=A\cdot\tr(h)+|\psi_{H}|^{2}_{H,\omega},
\end{equation}
where the constant $A$ is chosen large enough such that
\begin{equation}
	\begin{split}
		\bigg(\frac{\partial}{\partial t}-\Delta\bigg)f\leq &-\tilde{C}_{5}|\psi_{H}|_{H,\omega}^{2}+\tilde{C}_{6}.
	\end{split}
\end{equation}
All  the constants $\tilde{C}_{i}$ depend only on $C_{1}$, $H_{0}$ and the geometry of $(X,\omega)$. Then using the maximum principle, we obtain
 \begin{equation}
	\sup_{X\times[0,T)}|\psi_{H}|_{H,\omega}<C_{2}.
\end{equation}

\end{proof}

\section{Proof of theorem \ref{thm:m}}\label{sec:FP}
 Let $H(t)$ be a solution of the equation (\ref{HFeq1}) with initial metric $K$. Without loss of generality, we can assume  that the initial metric $K$ satisfies $\tr(\Phi(K))=0$. As the same as it before, we set $h(t)=K^{-1}H(t)$ and $s(t)=\log(h(t))$.

First, we prove the following lemma which will be used in the proof of theorem \ref{thm:m}.
\begin{lem}\label{l:1}
	Let $\pi\in L_{1}^{2}(X,\mbox{End}(E))$ be a projection map, and it satisfies
	\begin{equation}
		\pi^{2}=\pi^{*H}=\pi,\ \ (Id-\pi)D\pi=0\
	\end{equation}
	almost everywhere on $X$. Then $\pi\in C^{\infty}(X,\mbox{End}(E))$, and $\pi(E)$ defines a $D$-subbundle of $E$.
\end{lem}
\begin{proof}
	By the conditions above, we have $d|\pi|^{2}_{H}=d\tr(\pi^{2})=0$ in the sense of distribution. Therefore $|\pi|^{2}_{H} :=r_{1}\leq r$ almost everywhere, and $\pi\in L^{p}(X,\mbox{End}(E))$ for $0<p\leq+\infty$. The rest of the proof is exactly the same as Lemma \ref{l:0}.
\end{proof}

\begin{prop}\label{Prop:m}
  Suppose that the NHYM bundle $(E,D)$ is simple, then there exists a constant $C>0$ independent of $t$ such that
  \begin{equation}\label{PF1}
    \|s(t)\|_{L^{\infty}}<C.
  \end{equation}
\end{prop}
\begin{proof}
  Following from Simpson's arguments as in \protect{\cite[Proposition 5.3]{S1}}, we will prove if the equality (\ref{PF1}) doesn't hold, there exists a $D$-invariant subbundle destabilizing the simplicity. By Proposition \ref{Prop:Ineq} and Proposition \ref{HFP1}, we have
  \begin{equation*}
    ||s(t)||_{L^{\infty}}\leq C_1||s(t)||_{L^1}+C_2,
  \end{equation*}
  where $C_{1}, C_{2}$ are constants independent of $t$. So it is enough to show that $\|s(t)\|_{L^{1}}<C$.

  If not, there exists a  subsequence $t_i\to +\infty$ such that
  \begin{equation}\label{PF2}
    \lim_{t_{i}\rightarrow +\infty}\|s(t_{i})\|_{L^{1}}=+\infty.
  \end{equation}
  Set $u_i=l_{i}^{-1}s(t_i)$, where $l_i=||s(t_i)||_{L^1}$, then we have $||u_i||_{L^1}=1$. Moreover,
  \begin{equation}
    {\rm tr}u_{i}=0,\quad ||u_i||_{L^{\infty}}\leq C.
  \end{equation}
By the definition of Donaldson's functional, we have
 \begin{equation}
     \frac{d}{dt}M(K,H(t))=-16\int_{X}|\sqrt{-1}\Lambda_{\omega}G_{H(t)}|^{2}_{H(t)}\frac{\omega^{n}}{n!},
 \end{equation}
 so $M(K,H(t))<0$ for $t>0$. Hence
  \begin{equation}\label{PF3}
    \int_X{\rm tr}\bigg(\big(-4\sqrt{-1}\Lambda_{\omega}G_{K}\big)u_i\bigg)\frac{\omega^n}{n!}+l_i\int_{X}
    \langle \Psi(l_iu_i)(Du_i),Du_i\rangle_K\frac{\omega^n}{n!}< 0.
  \end{equation}
  Consider the following function
  \begin{equation}\label{PF4}
    l\Psi(lx,ly)=\begin{cases}
      \frac{e^{l(y-x)}-l(y-x)-1}{l(y-x)^{2}},&x\neq y,\\
      l/2,&x=y.
    \end{cases}
  \end{equation}
  According to $||u_i||_{L^{\infty}}\leq C$, we may assume that $(x,y)\in [-C,C]\times [-C,C]$. Then we can easy check that
  \begin{equation}\label{PF5}
    l\Psi(lx,ly)\to\begin{cases}
      \frac{1}{x-y},&x>y,\\
      +\infty,&x\leq y,
    \end{cases}
  \end{equation}
  increases monotonically as $l\to \infty$. Let $\zeta:\mathbb{R}\times \mathbb{R}\to \mathbb{R}$ be a positive smooth function satisfy $\zeta(x,y)\leq \frac{1}{x-y}$ whenever $x>y$. Then from (\ref{PF3}),(\ref{PF5}) and the arguments in \protect{\cite[Lemma 5.4]{S1}}, we obtain
  \begin{equation}\label{PF6}
    \int_X{\rm tr}\bigg(\big(-4\sqrt{-1}\Lambda_{\omega}G_K\big)u_i\bigg)\frac{\omega^n}{n!}+\int_{X}\langle
    \zeta(u_i)(Du_i),Du_i\rangle_K\frac{\omega^n}{n!}\leq 0,
  \end{equation}
  for $i\gg 0$. In particular, if we take $\zeta=\frac{1}{3C}$. It is obvious to check that $\frac{1}{3C}<\frac{1}{x-y}$ since $-C\leq x,y\leq C$ and $x>y$. This implies that
  \begin{equation*}
    \int_X{\rm tr}\bigg(\big(-4\sqrt{-1}\Lambda_{\omega}G_K\big)u_i\bigg)\frac{\omega^n}{n!}+\frac{1}{3C}\int_{X}|Du_i|^2_K\frac{\omega^n}{n!}\leq 0,
  \end{equation*}
  for $i\gg 0$. Then we have
  \begin{equation*}
    \int_{X}|Du_i|_K^2\frac{\omega^n}{n!}<+\infty.
  \end{equation*}
  Thus, $u_i$ are bounded in $L_1^2$. Choose a subsequence $\{u_{ij}\}$ such that $u_{ij}\rightharpoonup u_{\infty}$ weakly in $L_1^2$, still denoted by $u_i$ for simplicity. Noting that $L_1^2\hookrightarrow L^1$ is compact, we get
  \begin{equation*}
    1=||u_i||_{L^1}\to ||u_{\infty}||_{L^1}.
  \end{equation*}
  This means that $||u_{\infty}||_{L^1}=1$ and $u_{\infty}$ is non-trivial. By using (\ref{PF6}) and following the same discussion as in \protect{\cite[Lemma 5.4]{S1}}, we have
  \begin{equation}\label{PF7}
  \int_X{\rm tr}\bigg(\big(-4\sqrt{-1}\Lambda_{\omega}G_K\big)u_{\infty}\bigg)\frac{\omega^n}{n!}+\int_{X}\langle
    \zeta(u_{\infty})(Du_{\infty}),Du_{\infty}\rangle_K\frac{\omega^n}{n!}\leq 0.
  \end{equation}

By the same argument as in \cite{S1}, we can obtain the following lemma.
  \begin{lem}
  The eigenvalues of $u_{\infty}$ are constants almost everywhere.
  \end{lem}
  \begin{proof}
  It is enough to show that for any $\varphi:\mathbb{R}\rightarrow\mathbb{R}$, $\mbox{tr}(\varphi(u_{\infty}))$ is a constant almost everywhere. First, we have
\begin{equation}
d\tr(\varphi(u_{\infty}))=\tr(D\varphi(u_{\infty}))=\tr(d\varphi(u_{\infty})(Du_{\infty})).
\end{equation}
Suppose $N>0$ is a constant which is big enough. Choose $\Upsilon:\mathbb{R}\times \mathbb{R}\to \mathbb{R}$ such that
\begin{equation}
\Upsilon(x,x)=d\varphi(x,x)
\end{equation}
and
\begin{equation}
N\Upsilon^{2}(x,y)<\frac{1}{x-y}
\end{equation}
for $x>y$. Then
\begin{equation}
\tr(d\varphi(u_{\infty})(Du_{\infty}))=\tr(\Upsilon(u_{\infty})(Du_{\infty})).
\end{equation}
Let $\zeta=N\Upsilon^{2}$. By (\ref{PF7}), we get
\begin{equation}
\int_{X}|\Upsilon(u_{\infty})(Du_{\infty})|_{K}^{2}\frac{\omega^{n}}{n!}\leq \frac{4}{N}\int_{X}\tr(\sqrt{-1}\Lambda_{\omega}G_{K}u_{\infty})\frac{\omega^{n}}{n!}.
\end{equation}
Combining above, we have
\begin{equation}
\|d\tr(\varphi(u_{\infty}))\|_{L^{2}}^{2}\leq \frac{C}{N}.
\end{equation}
Let $N\rightarrow +\infty$, then we get $d\tr(\varphi(u_{\infty}))=0$.
  \end{proof}

  Next, we want to construct a $D$-invariant subbundle which contradicts the simplicity of $(E,D)$ by using Lemma \ref{l:1} about $L^2_1$ bundle. Let $\lambda_1<\cdots<\lambda_l$ be the distinct eigenvalues of $u_{\infty}$. Since ${\rm tr}u_{\infty}=0$ and $||u||_{L^1}=1$, we must have $l\geq 2$. For each $1\leq \alpha \leq l-1$, we define a function $P_{\alpha}$ such that
  \begin{equation*}
    P_{\alpha}=\begin{cases}
      1,&x\leq \lambda_{\alpha},\\
      0,&x\geq \lambda_{\alpha+1}.
    \end{cases}
  \end{equation*}
  Set $\pi_{\alpha}=P_{\alpha}(u_{\infty})$. Then from \protect{\cite[p.887]{UhYau}}, we have
    \begin{enumerate}
      \item[$(1)$] $\pi_{\alpha}\in L_1^2$;
      \item[$(2)$] $\pi_{\alpha}^2=\pi_{\alpha}=\pi_{\alpha}^{*K}$;
      \item[$(3)$] $({\rm Id}-\pi_{\alpha})D\pi_{\alpha}=0$.
    \end{enumerate}
  By Lemma \ref{l:1}, $\{\pi_{\alpha}\}_{\alpha=1}^{l-1}$ determine $l-1$ $D$-subbundles of $E$ denoted by $E_{\alpha}$. So we get a contradiction.
\end{proof}

\begin{proof}[Proof of Theorem \ref{thm:m}]
We have already showed that the existence of harmonic metrics implies semisimplicity. So we just need to show the converse is also true. It  comes from Proposition \ref{Prop:m}, Proposition \ref{p:last} and the regularity theory of elliptic equation.
\end{proof}

\medskip

{\bf  Acknowledgement:} The research was supported by the National Key R and D Program of China 2020YFA0713100.
The  authors are partially supported by NSF in China No.12141104, 11801535 and 11721101.


\end{document}